\documentclass[a4paper,12pt]{article}

\usepackage{amsthm}   \usepackage{amsfonts}     \usepackage{amssymb}
\usepackage{amsmath}  \usepackage{mathrsfs}     \usepackage{mathtools}
\usepackage{url}       \usepackage{fancyhdr}

\usepackage{relsize}
\usepackage[all]{xy}    \usepackage{amscd}   \usepackage{epsfig}

\usepackage{array}   \usepackage{multirow}  \usepackage{booktabs,caption,fixltx2e}
\usepackage[flushleft]{threeparttable}

\usepackage{cases} 
\usepackage{ulem}  
\usepackage{xcolor}  

\usepackage[T1]{fontenc} 
\usepackage[utf8]{inputenc}
\usepackage{authblk}  
\usepackage{lineno}

\newtheorem{theorem}{Theorem}
\newtheorem{lemma}[theorem]{Lemma}
\newtheorem{proposition}[theorem]{Proposition}
\newtheorem{corollary}[theorem]{Corollary}

\theoremstyle{definition}
\newtheorem{example}{Example}

\newtheorem{remark}{Remark}
\newtheorem*{case}{Case}

\newcommand{\Aut}{{\mathrm{Aut}}}

\newcommand{\Gal}{{\mathrm{Gal}}}

\begin{document}
\title{Complete regular dessins of odd prime power order}
\author[1,2]{Kan Hu\thanks{hukan@zjou.edu.cn}}
\author[3]{Roman Nedela\thanks{nedela@savbb.sk}}
\author[1,2]{Na-Er Wang\thanks{wangnaer@zjou.edu.cn}}
\affil[1]{School of Mathematics, Physics and Information Science, Zhejiang Ocean University,
Zhoushan, Zhejiang 316022, People's Republic of China}
\affil[2]{Key Laboratory of Oceanographic Big Data Mining \& Application of Zhejiang Province,
Zhoushan, Zhejiang 316022, People's Republic of China}
\affil[3]{Department of Mathematics, University of West Bohemia, NTIS FAV, Pilsen, Czech Republic}
\affil[4]{Mathematical Institute of Slovak Academy of Sciences, Bansk\'a Bystrica, Slovak Republic}

\maketitle

\begin{abstract}
A dessin  is a $2$-cell embedding of a connected $2$-coloured bipartite graph into an orientable closed surface.
A dessin is regular if its group of colour- and orientation-preserving automorphisms acts regularly on the edges.
In this paper we employ group-theoretic method to determine and enumerate the isomorphism classes of regular
dessins with the complete bipartite underlying graphs of odd prime power order.\\[2mm]
\end{abstract}

\section{Introduction}
A \textit{dessin} $\mathcal{D}$ is an embedding $i: X\hookrightarrow C$ of a $2$-coloured connected
bipartite graph
$X$ into an orientable closed surface $C$ such that each component of $C\setminus i(X)$ is homeomorphic
to an open disc. An \textit{automorphism} of a dessin $\mathcal{D}$ is a permutation of the edges
which preserves the graph structure and  vertex-colouring, and extends to an orientation-preserving
self-homeomorphism of the supporting surface.
The set of automorphisms forms the
automorphism group of $\mathcal{D}$ under composition, and it acts semi-regularly on the edges.
If this action is transitive, and hence regular, then the dessin is called \textit{regular} as well.

For a given dessin $\mathcal{D}$, following Lando and Zvonkin~\cite{LZ2004} the
\textit{reciprocal dessin} $\mathcal{D}^*$ of $\mathcal{D}$ is the dessin obtained
from $\mathcal{D}$ by interchanging the vertex colours of $\mathcal{D}$. It follows
that $\mathcal{D}^*$ has the same underlying graph, the same automorphism group and the
same supporting surface as $\mathcal{D}$. Thus dessins with a given bipartite graph $X$ appear in \textit{reciprocal pairs} $(\mathcal{D},\mathcal{D}^*)$. In particular, a regular dessin $\mathcal{D}$
is \textit{symmetric}  if $\mathcal{D}$ is isomorphic to $\mathcal{D}^*$. It follows that symmetric dessins possess an external
symmetry transposing the vertex colours and therefore may be viewed as (arc)regular bipartite maps.

It is well known that compact Riemann surfaces and complex projective algebraic curves are equivalent
categories. By the Riemann Existence Theorem~\cite[Theorem 1.8.14]{LZ2004} every dessin on an orientable
closed surface $C$ determines a complex structure on $C$ which makes it into a Riemann surface defined
over an algebraic number field. By Bely\v{\i}'s Theorem~\cite{Belyi1979} the inverse is also true:
if a curve $C$ can be defined over some algebraic number field, then its complex structure can be
obtained from a dessin on $C$. The absolute Galois group
$\mathbb{G}=\Gal(\mathbb{\bar Q}/\mathbb{Q})$ acts naturally on the curves defined
over $\mathbb{\bar Q}$, the field of all algebraic numbers. Grothendieck observed that
this action induces a faithful action of $\mathbb{G}$
on the associated dessins~\cite{Gro1997}. Gonz\'alez-Diez and Jaikin-Zapirain
have recently shown that this action remains faithful even when restricted to regular dessins~\cite{GJ2015}.
Thus one can study $\mathbb{G}$ through its action on such simple and symmetrical combinatorial objects as
regular dessins.

Regular dessins with complete bipartite underlying graphs will be called \textit{complete regular dessins}. Because the algebraic curves associated with complete regular dessins may be viewed as generalization of Fermat curves, it is important to study complete regular dessins and the associated curves~\cite{JS1996}.

 The classification of complete regular dessins is an open problem in this field, posed by Jones in~\cite{Jones2015}. In this direction, complete symmetric  dessins have been classified in a series of
papers~\cite{DJNS2007,DJNS2010,Jones2010,JNS2007,JNS2008, KK2005, KK2008, NSZ2002}. As for the general case, it has been shown that there exists a one-to-one correspondence  between the following three categories: the isomorphism classes of complete regular dessins, equivalence classes of exact bicyclic group factorisations and certain pairs of skew-morphisms of the cyclic groups~\cite{FHNSW}.

The correspondence between complete regular dessins and exact bicyclic group factorisations establishes a group-theoretic method to classify and enumerate the complete regular dessins. Following this method, it was shown in~\cite{FL2017} that the graph $K_{m,n}$ underlies a unique edge-transitive embedding if and only if $\gcd(m,\phi(n))=\gcd(n,\phi(m))=1$ where $\phi$ is the Euler's totient
function.

The symmetric and nonsymmetric complete regular dessins with underlying graphs $K_{p^e,p^e}$ have been classified in \cite{JNS2008} and \cite{CJSW2009}. Recently, the edge-transitive circular embeddings of $K_{p^e,p^f}$ have been determined in \cite{FLQ2018}, where a circular embedding is a map whose boundary walk of each face is a simple cycle without repeated edges. In this paper we consider the classification and enumeration of complete regular dessins of odd prime power order without any additional restriction. Our main result is the following
\begin{theorem}\label{MAIN}
 Let $p$ be an odd prime, and let $\nu(d,e)$ denote the number of isomorphism classes of reciprocal pairs of regular
dessins  with complete bipartite underlying graph $K_{p^d,p^e}$, $0\leq d\leq e$. Then
\begin{itemize}
 \item[\rm(i)]if $d=0$ then $\nu(d,e)=1$,
 \item[\rm(ii)]if $0<d=e$ then $\nu(d,e)=\frac{1}{2}p^{e-1}(1+p^{e-1})$,
 \item[\rm(iii)]if $0<d<e$ then $\nu(d,e)=p^{2d-1}$.
\end{itemize}
Moreover, in each case the dessins have type $(p^d,p^e,p^e)$ and genus $\frac{1}{2}(p^d-1)(p^e-2)$.
\end{theorem}
\begin{remark}
Dessins from (i) are exceptional, in the sense that their automorphism groups are cyclic,
whereas dessins from (ii) and (iii) have noncylcic metacyclic automorphism groups;
see Theorem~\ref{main}. Moreover,  up to isomorphism there are $p^{2(e-1)}$ complete regular dessins from (ii), and among them $p^{e-1}$ are symmetric. Thus by the results obtained in \cite{FLQ2018} we may conclude that all edge-transitive embeddings of $K_{p^e,p^e}$ are circular embeddings. Finally, the most interesting case is (iii) where we find that the number $\nu(d,e)$ is independent of the larger exponent parameter $e$.
\end{remark}
\section{Preliminaries}
In this section we collect some group-theoretic notations and results be to be used later, and
sketch the algebraic theory on regular dessins with complete bipartite underlying graphs.

Let $G$ be a finite $p$-group of exponent $\exp(G)=p^e$. For $0\leq i\leq e$ we define
\[
\mho_{i}(G)=\langle g^{p^i}\mid g\in G\rangle \quad\text{and}\quad \Omega_i(G)=\langle g\in G\mid g^{p^i}=1\rangle.
\]
The \textit{lower power series} and the  \textit{upper power series} of $G$ are defined to be
\[
G=\mho_0(G)\geq \mho_1(G)\geq\mho_2(G)\geq\cdots\geq\mho_e(G)=1.
\]
and
\[
1=\Omega_0(G)\leq\Omega_1(G)\leq\Omega_2(G)\leq \cdots\leq\Omega_e(G)=G.
\]

For a positive integer $i$, a finite $p$-group $G$ is called \textit{$p^i$-abelian} if
 for any $x,y\in G$,
 \[(xy)^{p^i}=x^{p^i}y^{p^i}.
 \]
 It is clear that $G$ is $p^i$-abelian if and only if the mapping $\pi_i:G\mapsto G, a\mapsto a^{p^i}$
is a group homomorphism, and that for positive integers $i<j$, if $G$ is $p^i$-abelian then it is also
$p^j$-abelian. Moreover, a finite $p$-group $G$ is called \textit{regular} if for any $x,y\in G$ there
exists $z\in\mho_1(\langle x,y\rangle')$ such that
\[
(xy)^{p}=x^py^pz
\]
 where $\langle x,y\rangle'$ denotes the derived subgroup of the group generated by $x$ and $y$.
The following result relates $p^i$-abelianness with regularity:
\begin{lemma}{\rm\cite[Lemma 5.1.7]{XQ2010}}\label{XU}
 Let $G$ be a finite $p$-group with $\mho_i(G')=1$ where $i$ is a positive integer. If  $G$ is regular,
then it is $p^i$-abelian.
\end{lemma}
Regular $p$-groups possess the following important property.
\begin{lemma}\label{PROPERTY}{\rm\cite[Theorem 10.5, III]{Huppert1967}}
 Let $G$ be a finite $p$-group. If $G$ is regular, then for any integer $i\geq1$ and any elements
$x,y\in G$, $(xy^{-1})^{p^i}=1$ if and only if $x^{p^i}=y^{p^i}$.
\end{lemma}

A finite group $G$ will be called $(m,n)$-\textit{bicyclic} if $G$ can be
factorised as a product $G=\langle a\rangle\langle b\rangle$ of two cyclic subgroups $\langle a\rangle$
and $\langle b\rangle$ of orders $m$ and $n$. In particular, if $\langle a\rangle\cap \langle b\rangle=1$,
then the factorisation $G=\langle a\rangle\langle b\rangle$  will be termed  \textit{exact}, and in these
circumstances, both the generating pair $(a,b)$ and the triple $(G,a,b)$ will be called
\textit{exact $(m,n)$-bicyclic}. Moreover, if $G$ has an automorphism transposing $a$ and $b$,
then the pair $(a,b)$ and the triple $(G,a,b)$ will be called \textit{exact $n$-isobicyclic}.
Two exact $(m,n)$-bicyclic triples $(G_i, a_i,b_i)$ $(i=1,2)$ will be called \textit{equivalent}
if the assignment $a_1\mapsto a_2$ and $b_1\mapsto b_2$ extends to a group isomorphism $G_1\to G_2$.

Given an exact $(m,n)$-bicyclic triple $(G,a,b)$, one may construct an $(m,n)$-complete regular dessin
as follows: First define a bipartite graph $X$ by taking the edges to be the
elements of $G$, and the black and white vertices to be the cosets of $g\langle a\rangle$ and
$g\langle b\rangle$ of $\langle a\rangle$ and $\langle b\rangle$ in $G$, with incidence given
by containment. Since $G=\langle a\rangle\langle b\rangle$, the cosets can be written as
\[
U=\{b^i\langle a\rangle\mid i=0,1,\ldots,n-1\}\quad\text{and}\quad V=\{a^j\langle b\rangle\mid j=0,1,\ldots,m-1\}.
\]
It is clear that $X$ is the complete bipartite graph $K_{m,n}$. Moreover,
the right multiplication of $a$ and $b$ defines a cyclic order
\[
(b^ia,b^ia^2,\ldots, b^ia^{m-1})\quad\text{and}\quad (a^jb,a^jb^2,\ldots, a^jb^{n-1})
\]
of the edges around the vertices $b^i\langle a\rangle$ and $a^j\langle b\rangle$ of $X$,
and these local orientations determine an embedding of $K_{m,n}$ into an oriented surface. The left
multiplication of $G$ induces a regular action of $G$ as a group of dessin automorphisms on the edges.
Conversely, every $(m,n)$-complete regular dessin arises in this way. In particular,
the $(m,n)$-complete regular dessin is symmetric if and only if the corresponding triple $(G,a,b)$ is
exact $n$-isobicyclic. To summarize we have
\begin{proposition}{\rm \cite{FHNSW}}\label{PROP1} The automorphism group of an $(m,n)$-complete regular
dessin is an exact $(m,n)$-bicyclic group, and for any exact $(m,n)$-bicyclic group $G$,  the
isomorphism classes of  $(m,n)$-complete regular dessins with automorphism group isomorphic to $G$
are in one-to-one correspondence with the equivalence classes of exact $(m,n)$-bicyclic generating
pairs of $G$. In particular, the symmetric $(n,n)$-complete regular dessins correspond to exact $n$-isobicyclic pairs of $G$.
\end{proposition}

\begin{proposition}\label{PROP2}
 Let $G$ be an exact $(m,n)$-bicyclic group, and  let $\mathcal{P}(G)$ and $\mathcal{B}(G)$
denote the sets of exact $(m,n)$-bicyclic pairs and exact $n$-isobicyclic pairs of $G$ if $m=n$, respectively.  Define
\[
\nu(G)=\frac{|\mathcal{P}(G)|}{|\Aut(G)|}
\quad\text{and}\quad\nu_0(G)=\frac{|\mathcal{B}(G)|}{|\Aut(G)|}.
 \]
 Then
\begin{itemize}
\item[\rm(i)]if $m=n$, then up to isomorphism $K_{n,n}$ underlies $\nu(G)$ regular dessins and
$\nu_0(G)$ symmetric regular dessins;
\item[\rm(ii)]if $m\neq n$,  then up to isomorphism $K_{m,n}$ underlies $\nu(G)$ reciprocal pairs
of regular dessins.
\end{itemize}
\end{proposition}
\begin{proof}We note that if $K_{m,n}$ underlies an $(m,n)$-complete regular dessin $\mathcal{D}$,
it also underlies an $(n,m)$-complete regular dessin $\mathcal{D}^*$, the reciprocal dessin of
$\mathcal{D}$ obtained by swapping the vertex colours of $\mathcal{D}$. Accordingly, every exact
$(m,n)$-bicyclic pair $(a,b)$ of $G$ corresponds to an exact $(n,m)$-bicyclic pair
$(b,a)$ of $G$. In particular, $\mathcal{D}$ is symmetric if and only if $\mathcal{D}\cong \mathcal{D}^*$,
or equivalently, the pair $(a,b)$ is exact $n$-isobicyclic.

 Since the action of $\Aut(G)$ on the set $\mathcal{P}(G)$ is semi-regular, by Proposition~\ref{PROP1}
the number $\nu(G)$ defined above is the number of isomorphism classes of $(m,n)$-complete regular dessins
with automorphism group isomorphic to $G$. If $m=n$, then $\nu(G)$ is precisely the number of isomorphism
classes of regular dessins with underlying graph $K_{n,n}$ and automorphism group isomorphic to $G$, of
which the number of  symmetric ones is $\nu_0(G)$. On the other hand, if $m\neq n$ then every $(m,n)$-complete
regular is not isomorphic to any $(n,m)$-complete regular dessin. Since complete regular dessins occur in
reciprocal pairs,  $\nu(G)$
is equal to the number of reciprocal pairs of regular dessins with underlying graphs $K_{m,n}$ and with
automorphism group isomorphic to $G$, as claimed.
\end{proof}
\begin{example}\label{UNIQUE}
For the abelian group defined by the presentation
\[
G=\langle a,b\mid a^m=b^n=[a,b]=1\rangle,
 \]
the triple $(G,a,b)$ determines a reciprocal pair of regular dessins with underlying graph $K_{m,n}$,
and every exact $(m,n)$-bicyclic triple of $G$ is equivalent to $(G,a,b)$.
Therefore, the graph $K_{m,n}$ underlies at least one reciprocal pair of regular dessins.
In \cite[Corollary 10]{FHNSW} it is shown that $K_{m,n}$ underlies a unique reciprocal pair of
regular dessins if and only if $\gcd(m,\phi(n))=\gcd(n,\phi(m))=1$; see also~\cite{FL2017}. In case of
$m=1$ the group $G$ is cyclic, and the underlying graph of the corresponding unique reciprocal pair of
complete regular dessins is the $n$-star $K_{1,n}$.
\end{example}

\section{Exact bicyclic $p$-groups}
In this section for odd primes $p$ we classify the automorphism groups of regular dessins  with complete
bipartite underlying graphs $K_{p^d,p^e}$.

The following result is fundamental.
\begin{lemma}{\rm\cite[Theorem 11.5, III]{Huppert1967}}\label{HUP}
Let $p>2$ be a prime. If $G $ is a finite bicyclic $p$-group, then $G$ is metacyclic.
\end{lemma}
The following H\"older's Theorem on metayclic group is well known.
\begin{lemma}{\rm\cite[Theorem 20]{Zass}}\label{Hold}
A metacyclic group $G$ with cyclic normal subgroup $N$ of order $n$ and with cyclic factor group
$F$ of order $m$ has two generators $x,y$ with the defining relations
\begin{align}\label{RELN}
G=\langle x,y\mid x^n=1,y^m=x^t,y^{-1}xy=x^r\rangle
\end{align}
and with the numerical conditions
\begin{equation}\label{NMM}
r^m\equiv 1\pmod{n}\quad\text{and}\quad t(r-1)\equiv0\pmod{n}.
\end{equation}
Conversely, for any positive integers $m,n,r$ and $t$ satisfying the numerical conditions in \eqref{NMM},
a metacyclic group of order $mn$ with the previous given properties
is defined by the three relations in \eqref{RELN}.
\end{lemma}

Note that if $p>2$ then every finite $p$-group with cyclic derived subgroup is
regular~\cite[Theorem 10.2, III]{Huppert1967}; thus by Lemma \ref{HUP} every finite bicyclic $p$-group
is a regular $p$-group.

\begin{lemma}
Let $p>2$ be a prime. If $G$ is metacyclic, then $G$ admits an exact bicyclic factorisation.
\end{lemma}
\begin{proof}
It is evident that every split metacyclic $p$-group admits an exact bicyclic factorisation, so it
suffices to consider non-spit metacyclic $p$-group $G$.
Without loss of generality we may assume that $G$ has a presentation
\[
G=\langle a,b\mid a^{p^l}=1, a^{p^m}=b^{p^n}, a^b=a^r\rangle,
\]
where $1\leq m<l$. By Proposition~\ref{PROPERTY} If $m\geq n$ then $ba^{-p^{m-n}}=1$ and
$G=\langle a\rangle\langle ba^{-p^{m-n}}\rangle$, whereas if $m<n$ then $ab^{-p^{n-m}}=1$
and $G=\langle b\rangle\langle ab^{-p^{n-m}}\rangle$. By the product formula in either case
the factorisation must be exact.
\end{proof}

The following technical results will be useful.
\begin{lemma}{\rm \cite{DH2017}}\label{EXP}
Let $G$ be a bicyclic $p$-group with a bicyclic factorisation $G=\langle\alpha\rangle\langle\beta\rangle$ where
$| \alpha|\geq|\beta|$. Then $\exp(G)=|\alpha|$. In particular if the factorisation is exact,
then for every exact bicyclic factorisation $G=\langle \alpha'\rangle\langle \beta'\rangle$ with
$| \alpha'|\geq|\beta'|$ we have $|\alpha'|=|\alpha|$ and $|\beta'|=|\beta|$.
\end{lemma}

\begin{lemma}{\rm\cite[lemma 9]{HNW2014}}\label{HNW}
Let $G=AB$ be a finite abelian group, then $G$ is cyclic if and only if $\gcd(|A|/|A\cap B|,|B|/|A\cap B|)=1$.
\end{lemma}
\begin{lemma}{\rm \cite[Lemma 6.2.3]{XQ2010}}\label{NUM}
Let $p$ be an odd prime, and $n$ a positive integer. Then the
multiplicative group $U(p^n)$ of $\mathbb{Z}_{p^n}$ is cyclic of order $\phi(p^n)=p^{n-1}(p-1)$, and its unique
 subgroup of order $p^i$ $(0\leq i\leq n-1)$ consists of elements of the form $kp^{n-i}+1$ where $k\in\mathbb{Z}_{p^i}$.
\end{lemma}

The following theorem classifies the automorphism groups of $(p^d,p^e)$-complete regular dessins for odd primes $p$.
\begin{theorem}\label{main}
Let $p>2$ be a prime, and let $d$ and $e$ be positive integers, $d\leq e$.
Then the automorphism group $G$ of a complete regular dessin
with underlying graph $K_{p^d,p^e}$
is isomorphic to one of the following groups:
\begin{itemize}
\item[\rm(i)]$\mathbf{M}_1(d,e,f)=\langle a,b\mid a^{p^e}=b^{p^d}=1, a^b=a^{1+p^f}\rangle$
where either $1\leq d\leq f\leq e\leq d+f$,
or $1\leq f<d<e\leq d+f$.
\item[\rm(ii)]$\mathbf{M}_2(d,e,f)=\langle a,b\mid a^{p^d}=b^{p^e}=1, a^b=a^{1+p^f}\rangle$
where $1\leq f<d<e.$
\item[\rm(iii)]$\mathbf{M}_3(d,e,h,f)=\langle a,b\mid a^{p^h}=1, b^{p^{d+e-h}}=a^{p^d},
a^b=a^{1+p^f}\rangle $ where $h-d\leq f<d<h<e$.
\end{itemize}
Moreover, the groups are pairwise non-isomorphic.
\end{theorem}
\begin{proof}
By Proposition \ref{PROP1} the group $G$ admits an exact $(p^d,p^e)$-factorisation
$G=AB$ where $A\cong\mathrm{C}_{p^d}$, $B\cong\mathrm{C}_{p^e}$ and $A\cap B=1$.
If $G$ is abelian then $G\cong\mathbf{M}_1(d,e,e)$. In what follows we
assume that $G$ is non-abelian. By Lemma~\ref{HUP} $G$ is metacyclic and so
regular~\cite[Theorem 10.2, III]{Huppert1967}.
Thus $G$ has a cyclic normal subgroup $N$ of order $p^h$ such that
the quotient group $G/N=\bar A\bar B$ is cyclic where $\bar A=AN/N$ and $\bar B=BN/N$.
By Lemma~\ref{EXP}, $\exp(G)=p^e$, so $h\leq e$. Since $\bar G$ is a $p$-group,
by Lemma~\ref{HNW} either $G=AN$ or $G=BN$.  We distinguish two cases:
\begin{case}[1]$G=AN$.\par
Since
\[
p^e\geq p^h=|N|\geq|N:N\cap A|=|AN:A|=|G:A|=|B|=p^e,
\]
we get $h=e$ and $|N\cap A|=1$, so $G=N\rtimes A\cong\mathrm{C}_{p^e}\rtimes\mathrm{C}_{p^d}$.
By H\"older's Theorem $G$ has a presentation
\[
G=\langle a,b_1\mid a^{p^e}=b_1^{p^d}=1, a^{b_1}=a^{r}\rangle,
\]
where $r^{p^d}\equiv1\pmod{p^e}$ and $r\not\equiv1\pmod{p^e}$. By Lemma~\ref{NUM} we may assume that $r=1+kp^f$
where $k\in\mathbb{Z}_{p^{e-f}}^*$, $1\leq f<e\leq d+f$.
Let $k'$ be the modular inverse of $k$ in $\mathbb{Z}_{p^{e-f}}$, then $a^{b_1^{k'}}=a^{(1+kp^f)^{k'}}=a^{1+p^f}$.
Replacing $b_1^{k'}$ with $b$ we obtain the groups in (i) where $1\leq f<e\leq d+f$.
\end{case}

\begin{case}[2]$G=BN$.\par
If $N\cap B=1$ then as before it is easy to show that $|N|=p^d$ and
$G=N\rtimes B\cong\mathrm{C}_{p^d}\rtimes\mathrm{C}_{p^e}$,
so $G\cong\mathbf{M}_2(d,e,f)$. In what follows we assume that $N\cap B>1$, then
\[
p^{e+d}=|G|=|BN|=|N||B|/|N\cap B|=p^{h+e}/|N\cap B|,
\]
so $|N\cap B|=p^{h-d}$ where $1\leq d<h\leq e$. Then $G$ has a presentation
\[
G=\langle a_1,b_1\mid a_1^{p^h}=b_1^{p^e}=1, b_1^{p^{d+e-h}}=a_1^{sp^d}, a_1^{b_1}=a_1^{1+kp^f}\rangle,
\]
where $s$ and $k$ are positive integers coprime to $p$, $1\leq f<h$. Replacing $a_1^{sk'}$ with $a$
and $b_1^{k'}$ with $b$ where $k'$ is the modular inverse of $k$ in $\mathbb{Z}_{p^{h-f}}$
the presentation of $G$ is transformed to the form
\[
G=\langle a,b\mid a^{p^h}=b^{p^e}=1, b^{p^{d+e-h}}=a^{p^d}, a^b=a^{1+p^f}\rangle.
\]
Since $a^{p^d}=(a^{p^d})^{b}=(a^{b})^{p^d}=a^{p^d(1+p^f)}$, we have $p^f\equiv0\pmod{p^{h-d}}$, and
so $h\leq d+f$. In what follows we distinguish three subcases:

If $h=e$, then $b^{p^d}=a^{p^d}$, so by Proposition~\ref{PROPERTY} $(a^{-1}b)^{p^d}=1$.
 Thus $G=\langle a,b\rangle=\langle a,a^{-1}b\rangle\cong\mathrm{C}_{p^e}\rtimes \mathrm{C}_{p^d}$,
which is a group in (i).

If $h<e$ and $d\leq f$, then
$G'=\langle [a,b]\rangle=\langle a^{p^f}\rangle\leq \langle a^{p^d}\rangle=\langle b^{p^{d+e-h}}\rangle\leq \langle b\rangle$,
so $\langle b\rangle\unlhd G$. By Proposition~\ref{PROPERTY} we have $(b^{-p^{e-c}}a)^{p^d}=1$, so
$G=\langle a,b\rangle=\langle b, b^{-p^{e-c}}a\rangle \cong \mathrm{C}_{p^e}\rtimes \mathrm{C}_{p^d}$, again a group in (i).

We are left with the subcase where $h<e$ and $f<d$, and we obtain the groups $ \mathbf{M}_3(d,e,h,f)$ in (iii).
\end{case}

 Finally, to determine the isomorphism relations between the groups in (i)--(iii), we summarize
  the invariant types of $G'$ and $G/G'$ in Table~\ref{TAB}.
\begin{center}
\begin{threeparttable}[b]
\caption{Invariant types of $G'$ and $G/G'$}\label{TAB}
\begin{tabular*}{117mm}[c]{|p{22mm}|p{10mm}|p{24mm}|p{44mm}|}
\toprule
Group                  &$G'$                    & $G/G'$          & Condition\\
\hline
$\mathbf{M}_1(d,e,f)$     & $\mathrm{C}_{p^{e-f}}$ & $\mathrm{C}_{p^f}\times\mathrm{C}_{p^d}$  & $1\leq d\leq f\leq e\leq d+f$ or $1\leq f<d<e\leq d+f$\\
\hline
$\mathbf{M}_2(d,e,f)$     & $\mathrm{C}_{p^{e-f}}$ & $\mathrm{C}_{p^f}\times\mathrm{C}_{p^e}$ & $1\leq f<d<e$\\
\hline
$\mathbf{M}_3(d,e,h,f)$     & $\mathrm{C}_{p^{h-f}}$ & $\mathrm{C}_{p^f}\times\mathrm{C}_{p^{e+d-h}}$  & $h-d\leq f<d<h<e$\\
\bottomrule
\end{tabular*}
\end{threeparttable}
\end{center}
Note that each of the groups has exponent $p^e$ and order $p^{d+e}$. It is easily seen from
the table that the groups from distinct families, or from the same family but with distinct parameters, are not isomorphic.
\end{proof}

\section{Exact bicyclic pairs and automorphisms}
In this section  for the groups $G$ determined in Theorem~\ref{main} we determine the set $\mathcal{P}(G)$
of exact $(p^d,p^e)$-bicyclic pairs $(\alpha,\beta)$ and the set $\Aut(G)$ of automorphisms of $G$.

The following technical result will be useful.
\begin{lemma}\label{TECH}
For the groups given in Theorem~\ref{main}, the following statements hold true:
\begin{itemize}
\item[\rm(i)]The group $G=\mathbf{M}_1(d,e,f)$ is $p^{e-f}$-abelian and $Z(G)=\langle a^{p^{e-f}}\rangle\langle b^{p^{e-f}}\rangle$.
\item[\rm(ii)]The group $G=\mathbf{M}_2(d,e,f)$ is $p^{d-f}$-abelian and $Z(G)=\langle a^{p^{d-f}}\rangle\langle b^{p^{d-f}}\rangle$.
\item[\rm(iii)]The group $G=\mathbf{M}_3(d,e,h,f)$ is $p^{h-f}$-abelian and $Z(G)=\langle a^{p^{h-f}}\rangle\langle b^{p^{h-f}}\rangle$.
\end{itemize}
\end{lemma}
\begin{proof}
To prove (i), we note that $G'=\langle a^{p^f}\rangle\cong \mathrm{C}_{p^{e-f}}$, so $\mho_{e-f}(G')=1$. Hence
by Lemma~\ref{XU} the group $G$ is $p^{e-f}$-abelian. It can be easily verified that
$Z(G)=\langle g^{p^{e-f}},h^{p^{e-f}}\rangle$. The proof for other cases is similar and is left to the reader.
\end{proof}
\begin{lemma}\label{TRIPLE}
 Let $\alpha=b^ia^j$ and $\beta=b^ka^l$ where $a,b$ are the generators of the group $G$  given in Theorem~\ref{main}, then
 \begin{itemize}
\item[\rm(i)]for $G=\mathbf{M}_1(d,e,f)$, if $(\alpha,\beta)$ is an exact $(p^d,p^e)$-bicyclic pair of $G$,
then either $d=e$ and $p\nmid(il-jk)$, or $d<e$,  $p^{e-d}\mid j$ and $p\nmid il$;
\item[\rm(ii)]for $G=\mathbf{M}_2(d,e,f)$, if $(\alpha,\beta)$ is an exact $(p^d,p^e)$-bicyclic pair of $G$,
then $p^{e-d}\mid i$ and $p\nmid jk$;
\item[\rm(iii)]for $G=\mathbf{M}_3(d,e,h,f)$, if $(\alpha,\beta)$ is an exact $(p^d,p^e)$-bicyclic pair of $G$,
then $p^{e-h}\| i$, $p^{h-d}\mid (j+i/p^{e-h})$ and $p\nmid jk$.
 \end{itemize}
 Conversely, in each case if the numerical conditions are satisfied, then $(\alpha,\beta)$ is an exact
$(p^d,p^e)$-bicyclic pair of $G$.
 \end{lemma}
\begin{proof}
By the product formula $G=\langle\alpha\rangle\langle\beta\rangle$
is an exact $(p^d,p^e)$-bicyclic factorisation of $G$ if and only if $|\alpha|=p^d$, $|\beta|=p^e$ and
$\langle \alpha\rangle\cap\langle\beta\rangle=1$. Note that if $(\alpha,\beta)$ is an exact $(p^d,p^e)$-bicyclic
pair of $G$, then by Burnside's Basis Theorem $il-jk\not\equiv0\pmod{p}$.

For $G=\mathbf{M}_1(d,e,f)$, if $d=e$ then the result is obvious; see also \cite{CJSW2009}. Now consider the
case $d<e$. First assume that $(\alpha,\beta)$ is an exact $(p^d,p^e)$-bicyclic pair of $G$, then by
Lemma~\ref{TECH}(i) $G$ is $p^{e-f}$-abelian.
Since $e-f\leq d\leq e $, by Proposition~\ref{XU} we have
\[
1=\alpha^{p^d}=(b^ia^j)^{p^d}=b^{ip^d}a^{jp^d}=a^{jp^d},
\]
so $j\equiv0\pmod{p^{e-d}}$, and hence $il\not\equiv0\pmod{p}$. Conversely, assume the congruences, then
it can
 be easily verified that $|\alpha|=p^d$ and $|\beta|=p^e$. It suffices to prove that $\langle \alpha\rangle\cap\langle \beta\rangle=1$.
 This is equivalent to that $\langle \alpha^{p^{d-1}}\rangle\cap\langle \beta^{p^{e-1}}\rangle=1$. Since
$e-f\leq e-1$,
 by Lemma~\ref{TECH}(i) we have $\beta^{p^{e-1}}=(b^ka^l)^{p^{e-1}}=a^{lp^{e-1}}\in \langle a\rangle$. Since $\alpha^{p^{d-1}}=(b^ia^j)^{p^{d-1}}=b^{ip^{d-1}}a^{js}$ for some integer $s$, we have $\alpha^{p^{d-1}}\notin\langle a\rangle$, so
 $\langle \alpha^{p^{d-1}}\rangle\cap\langle \beta^{p^{e-1}}\rangle=1$.

The proof for the group $G=\mathbf{M}_2(d,e,f)$ is similar and omitted.

For $G=\mathbf{M}_3(d,e,h,f)$, if $G=\langle\alpha\rangle\langle\beta\rangle$
is an exact $(p^d,p^e)$-bicyclic factorisation of $G$, then
$il-jk\not\equiv0\pmod{p}$. By Lemma~\ref{TECH}(iii) $G$ is $p^{h-f}$-abelian. Since $h-f<d$,
we have
\[
1=\alpha^{p^d}=(b^ia^j)^{p^d}=b^{ip^d}a^{jp^d}=b^{p^d(i+jp^{e-h})},
\]
and so $i+jp^{e-h}\equiv0\pmod{p^{e-d}}$. Since $0<e-h<e-d$, we obtain $p^{e-h}\|i$ and
$j+i/p^{e-h}\equiv0\pmod{p^{h-d}}$. Since $il-jk\not\equiv0\pmod{p}$, $jk\not\equiv0\pmod{p}$.
Conversely, assume the numerical conditions, then it is easy to verify that $|\alpha|=p^d$ and $|\beta|=p^e$.
It suffices to prove $\langle \alpha\rangle\cap\langle \beta\rangle=1$, or equivalently $\langle \alpha^{p^{d-1}}\rangle\cap\langle \beta^{p^{e-1}}\rangle=1$. Since $G$ is $p^{h-f}$-abelian and $h-f\leq d<h<e$, we have
\[
\beta^{p^{e-1}}=(b^ka^l)^{p^{e-1}}=b^{kp^{e-1}}a^{lp^{e-1}}=b^{kp^{e-1}}=(b^{d+e-h})^{kp^{h-d-1}}=a^{kp^{h-1}}\in\langle a\rangle.
\]
 Note that $p^{e-h}\|i$. Set $i=i'p^{e-c}$ where $p\nmid i'$. Then
 \[
 \alpha^{p^{d-1}}=(b^ia^j)^{p^{d-1}}\in b^{ip^{d-1}}\langle a\rangle=b^{i'p^{d+e-c-1}}\langle a\rangle.
  \]
  Since $\langle a\rangle\cap\langle b\rangle= \langle b^{p^{d+e-c}}\rangle$, we have
$\alpha^{p^{d-1}}\notin\langle a\rangle$. Thus $\langle \alpha^{p^{d-1}}\rangle\cap\langle \beta^{p^{e-1}}\rangle=1$, as required.
\end{proof}

\begin{corollary}\label{SPAIR1}
\begin{itemize}
\item[\rm(i)]The number of  exact $(p^d,p^e)$-bicyclic pairs of the group $\mathbf{M}_1(d,e,f)$
is equal to $p^{4e-3}(p^2-1)(p-1)$ if $d=e$, and
 $p^{3d+e-2}(p-1)^2$ if $d<e$.
\item[\rm(ii)]The number of  exact $(p^d,p^e)$-bicyclic pairs of the group $\mathbf{M}_2(d,e,f)$ is
 equal to $p^{3d+e-2}(p-1)^2$.
\item[\rm(iii)]The number of  exact $(p^d,p^e)$-bicyclic pairs of the group $\mathbf{M}_3(d,e,h,f)$
is equal to $p^{3d+e-2}(p-1)^2$.
\end{itemize}
\end{corollary}
\begin{proof}
 For $G=\mathbf{M}_1(d,e,f)$, by Lemma~\ref{TRIPLE} if $d=e$ then $(b^ia^j,b^ka^l)$ is a $(p^e,p^e)$-bicyclic
 pair of $G$ if and only if the matrix $\begin{pmatrix}i&j\\k&l\end{pmatrix}$ is invertible in the ring
$\mathbb{Z}_{p^e}$, so there are $|\mathrm{GL}(2,p^e)|=p^{4e-3}(p^2-1)(p-1)$ such pairs. On the other hand,
if $d<e$ then $(b^ia^j,b^ka^l)$ is a $(p^d,p^e)$-bicyclic pair of $G$ if and only if $p^{e-d}\mid j$ and
$p\nmid il$, in which case the number of choices for each $i,j,k$ and $l$ is $\phi(p^d)$, $p^d$, $p^d$ and
$\phi(p^e)$, respectively. Multiplying these gives the desired number in (i). The proof for other cases is similar.
\end{proof}
\begin{lemma}\label{CLAIM}
Let $p$ be an odd prime and $n$ a positive integer, then $p^{n-k+2}$ divides ${p^n\choose k}$
for all $3\leq k\leq n+2$.
\end{lemma}
\begin{proof}
First it can be easily proved by induction on $k$ that $k\le p^{k-2}$  for all integers $k\geq 3$. Now write $k=p^ek_1$ where $e\geq 0$ and $p\nmid k_1$, and denote the factor $p^e$ of $k$ by $(k)_p.$
   Since
\begin{align*}
{p^n\choose k}&=\frac{p^n(p^n-1)(p^n-2)\cdots [p^n-(k-1)]}{1\times 2\times\cdots\times
(k-1)\times k}=\frac{p^n}{k}\prod_{j=1}^{k-1}\frac{p^n-j}{j},
\end{align*}
and each of the numbers $\frac{p^n-j}{j}$ $(1\leq j\leq k-1)$ is coprime to $p$, we have ${p^n\choose k}_p=\frac{p^n}{(k)_p}$.
Since  $(k)_p\leq k\le p^{k-2}$, we get  $p^{n-k+2}\mid{p^n\choose k}$, as claimed.
\end{proof}
In what follows we determine the automorphisms of the groups in Theorem~\ref{main}.
\begin{lemma}\label{AUTO1}
Let $G=\mathbf{M}_1(d,e,f)$, then the assignment $a\mapsto b^ra^s, b\mapsto b^ta^u$ extends to an automorphism
of $G$ if and only if one of the following cases occur:
\begin{itemize}
\item[\rm(i)]$f=e=d$ and $ru-st\not\equiv0\pmod{p}$.
\item[\rm(ii)]$f<e=d$, $r\equiv0\pmod{p^{e-f}}$, $t\equiv1\pmod{p^{e-f}}$ and $s\not\equiv0\pmod{p}$.
\item[\rm(iii)]$d<f=e$, $u\equiv0\pmod{p^{e-d}}$ and $st\not\equiv0\pmod{p}$.
\item[\rm(iv)]$d\leq f<e$, $u\equiv0\pmod{p^{e-d}}$, $t\equiv1\pmod{p^{e-f}}$ and $s\not\equiv0\pmod{p}$.
\item[\rm(v)]$f<d<e$, $r\equiv0\pmod{p^{d-f}}$, $u\equiv0\pmod{p^{e-d}}$, $t\equiv1\pmod{p^{e-f}}$ and $s\not\equiv0\pmod{p}$.
\end{itemize}
\end{lemma}
\begin{proof}
Set $a_1=b^ra^s$ and $b_1=b^ta^u$. Then the assignment $a\mapsto a_1, b\mapsto b_1$ extends to an automorphism
 of $G$ if and only if  $a_1^{p^e}=b_1^{p^d}=1$, $a_1^{b_1}=a_1^{1+p^f}$ and $G=\langle a_1,b_1\rangle$.

 First assume that $a\mapsto a_1, b\mapsto b_1$ extends to an automorphism
 of $G$. Set  $q:=1+p^f$, then $a_1^{b_1}=a_1^{q}$. Since
\begin{align*}
a_1^{b_1}=(b^ra^s)^{b^ta^u}=(b^r)^{a^u}(a^s)^{b^t}=b^ra^{u(1-q^r)+sq^t}
\end{align*}
and
\[
a_1^{q}=(b^ra^s)^{q}=b^{rq}a^{s\sigma},
\]
where $\sigma=\sum_{i=1}^qq^{r(i-1)}$, we get $b^{r(q-1)}=a^{u(1-q^r)+s(q^t-\sigma)}$.
But $\langle a\rangle\cap\langle b\rangle=1$, thus
\begin{align}
&r(q-1)\equiv0\pmod{p^{d}},\label{NEQ1}\\
&u(q^r-1)\equiv s(q^t-\sigma)\pmod{p^e}.\label{NEQ2}
\end{align}
Note that $G=\langle a_1,b_1\rangle$, so by Burnside's Basis Theorem $ru-st\not\equiv0\pmod{p}$.
To simplify the numerical conditions we distinguish two cases:
\begin{case}[I]$d=e$. We distinguish two subcases:

If $f=d=e$ then $q=1$ and the congruences \eqref{NEQ1} and \eqref{NEQ2} are redundant.

If $f<d=e$ then by \eqref{NEQ1} we have $r\equiv0\pmod{p^{e-f}}$, so $st\not\equiv0\pmod{p}$ and
by Lemma~\ref{NUM} we get $q^r-1=(1+p^f)^r-1\equiv0\pmod{p^e}$. Setting $z=q^r-1$, then
 \begin{align}\label{NEQ123}
\sigma =\frac{q^{rq}-1}{q^r-1}=\frac{1}{z}((1+z)^q-1)={q\choose 1}+\sum_{i=2}^q{q\choose i}z^{i-1}.
 \end{align}
Thus $\sigma\equiv q\pmod{p^e}$ and hence \eqref{NEQ2} is reduced to $sq(q^{t-1}-1)\equiv0\pmod{p^e}$,
which implies that $t\equiv1\pmod{p^{e-f}}$ by Lemma~\ref{NUM}.
\end{case}

\begin{case}[II]$d<e$. Since $|b_1|=p^d$ we have $u\equiv0\pmod{p^{e-d}}$ and $st\not\equiv0\pmod{p}$. In what follows we
distinguish three subcases:

If $d<f=e$ then $q=1$ and so \eqref{NEQ1} and \eqref{NEQ2} are redundant.

If $d\leq f<e$ then $e+f-d\geq e$, so $u(q^r-1)\equiv0\pmod{p^e}$. As before write $z=q^r-1$
and expand $\sigma$ as \eqref{NEQ123}.
Note that $z\equiv0\pmod{p^f}$. Since $2f\geq d+f\geq e$, we have ${q\choose i}z^{i-1}\equiv0\pmod{p^e}$
for all $i\geq 2$, and hence  $\sigma\equiv q\pmod{p^e}$. Therefore \eqref{NEQ2} is reduced to
$sq(q^{t-1}-1)\equiv0\pmod{p^e}$, which implies that $t\equiv1\pmod{p^{e-f}}$.

Finally, if $f<d<e$, then by \eqref{NEQ1} $r\equiv0\pmod{p^{d-f}}$, so by Lemma~\ref{NUM} we get
$q^r-1\equiv0\pmod{p^d}$, and hence $u(q^r-1)\equiv0\pmod{p^e}$. As before write $z=q^r-1$ and expand $\sigma$ as \eqref{NEQ123}.
Then ${q\choose i}z^{i-1}\equiv0\pmod{p^e}$ for all $i\geq 2$, so $\sigma\equiv q\pmod{p^e}$.
Therefore \eqref{NEQ2} is reduced to $sq(q^{t-1}-1)\equiv0\pmod{p^e}$, which implies that $t\equiv1\pmod{p^{e-f}}$.
\end{case}

Conversely, in each case if the numerical conditions are fulfilled, then it is straightforward to verify
that the above assignment extends to an automorphism of $G$, as required.
\end{proof}

\begin{corollary}\label{SAUTO1}
Let $\mathbf{M}_1(d,e,f)$ be the group given by Theorem~\ref{main}. Then
\[
|\Aut(\mathbf{M}_1(d,e,f))|=
\begin{cases}
p^{4e-3}(p^2-1)(p-1),&f=e=d,\\
p^{2(e+f)-1}(p-1),& f<e=d,\\
p^{3d+e-2}(p-1)^2,& d<f=e,\\
p^{3d+f-1}(p-1),&d\leq f<e,\\
p^{2(d+f)-1}(p-1),&f<d<e.
\end{cases}
\]
\end{corollary}
\begin{proof}
If $f=e=d$, then by Lemma~\ref{AUTO1}(i) the size of $\Aut(\mathbf{M}_1(e,e,e))$ is
equal to the number of invertible matrices $\begin{pmatrix}
 r&s\\
t&u
\end{pmatrix}
$
with entries in $\mathbb{Z}_{p^e}$, which is $p^{4e-3}(p^2-1)(p-1)$.

In what follows the size of $\Aut(\mathbf{M}_1(d,e,f))$
for the remaining cases is determined by multiplying the numbers of choices for
the parameters $r,s,t$
and $u$. By Lemma~\ref{AUTO1}(ii)--(v), we have the following:

If $f<e=d$,  then the number of choices for each $r,s,t$
and $u$ is $p^f$, $\phi(p^e)$, $p^f$ and $p^e$, respectively.

If $d<f=e$,  then  the number of choices for each $r,s,t$
and $u$ is $p^d $, $\phi(p^e)$, $\phi(p^d)$ and $p^d$, respectively.

If $d\leq f<e$,  then  the number of choices for each $r,s,t$
and $u$ is $p^d$, $\phi(p^e)$, $p^{d+f-e}$ and $p^d$, respectively.

If $f<d<e$,  then the number of choices for each $r,s,t$
and $u$ is $p^f$, $\phi(p^e)$, $p^{d+f-e}$ and $p^d$, respectively.
\end{proof}

\begin{lemma}\label{AUTO2}
 Let $\mathbf{M}_2(d,e,f)$ and $\mathbf{M}_3(d,e,h,f)$ be the groups given by Theorem~\ref{main}, then
 \begin{itemize}
 \item[\rm(i)] the assignment $a\mapsto b^ra^s, b\mapsto b^ta^u$ extends to an automorphism of $\mathbf{M}_2(d,e,f)$
if and only if $r\equiv0\pmod{p^{e-f}}$, $s\not\equiv0\pmod{p}$ and $t\equiv1\pmod{p^{d-f}}$.
\item[\rm(ii)]the assignment $a\mapsto b^ra^s, b\mapsto b^ta^u$ extends to an automorphism of $\mathbf{M}_3(d,e,h,f)$
if and only if $r\equiv0\pmod{p^{d+e-h-f}}$, $s\equiv 1+up^{e-h}\pmod{p^{h-d}}$, $t\equiv1\pmod{p^{d-f}}$ and
$r\equiv(t-1)p^{e-h}\pmod{p^{e-f}}$.
\end{itemize}
\end{lemma}
\begin{proof}Denote $a_1=b^ra^s$ and $b_1=b^ta^u$, and write $q=1+p^f$.

(i)~First assume that the assignment $a\mapsto a_1,b\mapsto b_1$ extends to an automorphism of $G=\mathbf{M}_2(d,e,f)$,
then in terms of $a_1$ and $b_1$ the group $G$ has the presentation
 \[
 G=\langle a_1,b_1\mid a_1^{p^d}=b_1^{p^e}=1, a_1^{b_1}=a_1^{q}\rangle.
 \]
By Lemma~\ref{TECH}(ii) the group $G$ is $p^{d-f}$-abelian, so $ 1=a_1^{p^d}=b^{rp^d}a^{sp^d}=a^{sp^d}$,
and hence $r\equiv0\pmod{p^{e-d}}$.  By Burnside's Basis Theorem, $ru-st\not\equiv0\pmod{p}$, so
$st\not\equiv0\pmod{p}$. Moreover,
\begin{align*}
a_1^{b_1}&=(b^ra^s)^{b^ta^u}=(b^r)^{a^u}(a^s)^{b^t}=b^r[b^r,a^u](a^s)^{b^t}\in b^r\langle a\rangle
\end{align*}
and
\[
a_1^{q}=(b^ra^s)^{q}\in b^{rq}\langle a\rangle,
\]
so from the relation $a_1^{b_1}=a_1^{q}$ and the fact $\langle a\rangle\cap\langle b\rangle=1$
we deduce that $b^{rq}=b^r$. Thus $r(q-1)\equiv0\pmod{p^e}$, and hence $r\equiv0\pmod{p^{e-f}}$. By Lemma~\ref{TECH}(ii),
  $b^r\in Z(G)$, so $b^ra^{sq^t}=b_1^{a_1}=(a_1)^q=b^{rq}a^{sq},$ and hence $b^{r(q-1)}=a^{sq(q^{t-1}-1)}$.
  Since $\langle a\rangle\cap\langle b\rangle=1$, we get $sq(q^{t-1}-1)\equiv0\pmod{p^d}$.
  Therefore $t\equiv1\pmod{p^{d-f}}$.

  Conversely, if the numerical conditions are fulfilled,
  then it is straightforward to verify that the assignment extends to an automorphism of $G$.

(ii)~First assume that the assignment $a\mapsto a_1, b\mapsto b_1$ extends to
 an automorphism of $G=\mathbf{M}_3(d,e,h,f)$, then
 \[
 G=\langle a_1,b_1\mid a_1^{p^h}=1, b_1^{p^{d+e-h}}=a_1^{p^d}, a_1^{b_1}=a_1^{q}\rangle,
 \]
 where $q:=1+p^f$. By Lemma~\ref{TECH}(iii)
$G$ is $p^{h-f}$-abelian. Since $h-f\leq d<h$, we have
\[
1=a_1^{p^h}=(b^ra^s)^{p^h}=b^{rp^h}a^{sp^h}=b^{rp^h},
\]
so $r\equiv0\pmod{p^{e-h}}$. Since $ru-st\not\equiv0\pmod{p}$, we obtain $st\not\equiv0\pmod{p}$.

Moreover, we have
\begin{align*}
&a_1^{b_1}=(b^ra^s)^{b^ta^u}=(b^r)^{a^u}(a^s)^{b^t}=b^ra^{u(1-q^r)+sq^t},\\
&a_1^q=(b^ra^s)^q=b^{rq}a^{s\sigma},
\end{align*}
 where $\sigma=\sum_{i=1}^qq^{r(i-1)}$, so from the relation $b_1^{p^{d+e-h}}=a_1^{p^d}$ we deduce that
 \begin{align}\label{REL1}
b^{r(q-1)}=a^{u(1-q^r)+s(q^t-\sigma)}.
  \end{align}
Since $\langle a\rangle\cap\langle b\rangle=\langle b^{p^{d+e-h}}\rangle=\langle a^{p^d}\rangle$, we get
\begin{align}
r(q-1)&\equiv0\pmod{p^{d+e-h}},\label{REL2}\\
 u(q^r-1)&\equiv s(q^t-\sigma)\pmod{p^d}.\label{REL3}
 \end{align}
Upon substitution \eqref{REL1} is transformed to
\[
a^{u(1-q^r)+s(q^t-\sigma)}=b^{r(q-1)}=(b^{p^{d+e-h}})^{rp^{f+h-e-d}}=a^{rp^{f+h-e}},
\]
which implies that
\begin{align}\label{REL4}
u(1-q^r)+s(q^t-\sigma)\equiv rp^{f+h-e}\pmod{p^h}.
\end{align}
 By \eqref{REL2} we get $r\equiv0\pmod{p^{d+e-h-f}}.$ Setting $z=q^r-1$, then
 \[
 z=(1+p^f)^r-1={r\choose1}p^f+\sum_{i=2}^r{r\choose i}p^{if}.
 \]
 For all $i\geq 2$ by Lemma~\ref{CLAIM}  we have $p^{d+e-h-f-i+2+if}\mid {r\choose i}p^{if}$; since
 $ d+e-h-f-i+2+if\geq d-f-(i-2)+if=d+f+(i-2)(f-2)\geq d+f\geq h,$
 we get ${r\choose i}p^{if}\equiv0\pmod{p^h}$, and so $z=q^r-1\equiv rp^f\pmod{p^h}$. It follows that
 \begin{align*}
 \sigma=\frac{q^{rq}-1}{q^r-1}=\frac{1}{z}((1+z)^q-1)={q\choose1}+{q\choose2}z+\sum_{i=3}^q{q\choose i}z^{i-1}\equiv q\pmod{p^h}
 \end{align*}
 and so \eqref{REL3} and  \eqref{REL4} are reduced to
\begin{align}
sq(q^{t-1}-1)&\equiv0\pmod{p^d},\label{REL5}\\
sq(q^{t-1}-1)&\equiv rp^{f+h-e}(1+up^{e-h})\pmod{p^h}\label{REL6}.
\end{align}
By Lemma~\ref{NUM} we deduce from \eqref{REL5} that $t\equiv1\pmod{p^{d-f}}$. Hence
\[
q^{t-1}-1=(1+p^f)^{t-1}-1={t-1\choose 1}p^f+\sum_{i=2}^{t-1}{t-1\choose i}p^{if}\equiv (t-1)p^f\pmod{p^h},
\]
and consequently, \eqref{REL6} is reduced to
\begin{align}\label{REL7}
sq(t-1)p^f\equiv rp^{f+h-e}(1+up^{e-h})\pmod{p^h}.
\end{align}

We proceed to consider the relation $b_1^{p^{d+e-h}}=a_1^{p^d}$. By Lemma~\ref{TECH}(iii) $G$ is $p^{h-f}$-abelian.
Since  $h-f\leq d\leq d+e-h$, we have
$b_1^{p^{d+e-h}}=(b^ta^u)^{p^{d+e-h}}=b^{tp^{d+e-h}}a^{up^{d+e-h}} $ and $a_1^{p^d}=(b^ra^s)^{p^d}=b^{rp^d}a^{sp^d}$,
and so
\[
a^{(s-up^{e-h})p^d}=b^{tp^{d+e-h}-rp^d}=(b^{p^{d+e-h}})^{t-rp^{h-e}}=a^{(t-rp^{h-e})p^d}.
\]
Thus
\begin{align}\label{REL8}
(s-up^{e-h})p^d\equiv (t-rp^{h-e})p^d\pmod{p^h}.
\end{align}
Write $r=r_1p^{d+e-h-f}$ and $t=1+t_1p^{d-f}.$ Recall $q=1+p^f$ and $h-d\leq f$. Then \eqref{REL7} and \eqref{REL8} are reduced to
\begin{align}
&st_1\equiv r_1(1+up^{e-h})\pmod{p^{h-d}},\label{REL9}\\
&s\equiv 1+up^{e-h}+(t_1-r_1)p^{d-f}\pmod{p^{h-d}}.\label{REL10}
\end{align}
Substituting $1+up^{e-h}+(t_1-r_1)p^{d-f}$ for $s$ in \eqref{REL9} we obtain $(t_1-r_1)(1+up^{e-h}+p^{d-f})\equiv0\pmod{p^{h-d}}$,
thus $r_1\equiv t_1\pmod{p^{h-d}}$ (or equivalently, $r\equiv (t-1)p^{e-h}\pmod{p^{e-f}}$). Therefore \eqref{REL10} is reduced to $s\equiv 1+up^{e-h}\pmod{p^{h-d}}$.

Conversely, if the numerical conditions are fulfilled, then it is straightforward to verify that the above assignment extends to
an automorphism of $G$, as required.
\end{proof}

\begin{corollary}\label{SAUTO2}
Let $\mathbf{M}_2(d,e,f)$ and $\mathbf{M}_3(d,e,h,f)$ be the groups given by Theorem~\ref{main}. Then
\begin{itemize}
\item[\rm(i)]$|\Aut(\mathbf{M}_2(d,e,f))|=p^{d+e+2f-1}(p-1)$.
\item[\rm(ii)] $|\Aut(\mathbf{M}_3(d,e,h,f))|=p^{2d+e+2f-h}$.
\end{itemize}
\end{corollary}
\begin{proof}
(i)~By Lemma~\ref{AUTO2}(i), the number of choices for each $r,s,t$ and $u$ is $p^f$, $p^{d-1}(p-1)$, $p^{e+f-d}$
and $p^d$, respectively, and multiplying these gives the desired number.

(ii)~Note that $r,t\in\mathbb{Z}_{p^{d+e-h}}$ and $s,u\in\mathbb{Z}_{p^h}$.
By Lemma~\ref{AUTO2}(ii), we may write $r=r_1p^{d+e-h-f}$ and $t=1+t_1p^{d-f}$ where
$r_1\in\mathbb{Z}_{p^f}$ and $t_1\in\mathbb{Z}_{p^{e+f-h}}$, so the
congruence $r\equiv(t-1)p^{e-h}\pmod{p^{e-f}}$ is reduced to $r_1\equiv t_1\pmod{p^{h-d}}$.
Thus for each
$r_1\in\mathbb{Z}_{p^f}$ the number of choices for $t_1\in\mathbb{Z}_{p^{e+f-h}}$ such that
$r_1\equiv t_1\pmod{p^{h-d}}$ is equal to $p^{d+e+f-2h}$, and for each $u\in\mathbb{Z}_{p^h}$ the
number of choices for $s\in\mathbb{Z}_{p^h}$ such that $s\equiv 1+up^{e-h}\pmod{p^{h-d}}$ is equal
to $p^d$. Consequently, the desired number is the product $p^fp^{d+e+f-2h}p^hp^{d}=p^{2d+e+2f-h}.$
\end{proof}

\section{Enumeration}
In this section we calculate the number of isomorphism classes of reciprocal pairs of $(p^d,p^e)$-complete regular dessins.

The following result deals with the particular case $d=e$ where symmetric dessins may appear.
\begin{lemma}\label{ENM0}
For each $e\geq1$, up to isomorphism there are $p^{2(e-1)}$ regular dessins with underlying graphs $K_{p^e,p^e}$,
of which the number of  symmetric ones is $p^{e-1}$.
\end{lemma}
\begin{proof}
By Theorem~\ref{main} the automorphism group of a $(p^e,p^e)$-complete regular dessin
is isomorphic to $G:=\mathbf{M}_1(e,e,f)$ for some integer $f$, where $1\leq f\leq e$. By Corollary~\ref{SPAIR1}
the number of exact $(p^e,p^e)$-bicyclic triples of $G$ is $p^{4e-3}(p^2-1)(p-1)$,
and by Corollary~\ref{SAUTO1}, $|G|=p^{4e-3}(p^2-1)(p-1)$ if $f=e$,
and $|\Aut(G)|=p^{2(e+f)-1}(p-1)$ if $f<e$. Thus by Proposition~\ref{PROP2}
for each fixed $f$, up to isomorphism the number of $(p^e,p^e)$-complete regular dessins with
automorphism group isomorphic to $G$ is $1$ if $f=e$, and $p^{2e-2f-2}(p^2-1)$
if $1\leq f<e$. Summing up we obtain the total number of $(p^e,p^e)$-complete
regular dessins (up to isomorphism):
\[
1+\sum_{f=1}^{e-1}p^{2e-2f-2}(p^2-1)=p^{2(e-1)}.
 \]
 By~\cite[Theorem 1]{JNS2008} exactly $p^{e-1}$ of these are symmetric, as claimed.
\end{proof}
Combining Proposition~\ref{PROP2} with Corollary~\ref{SPAIR1}, \ref{SAUTO1} and \ref{SAUTO2} we immediately
obtain the following three results.
\begin{lemma}\label{ENM1}
 Let $d<e$, then for each fixed $f$, up to isomorphism the number $\nu_1(d,e,f)$ of reciprocal pairs of
$(p^d,p^e)$-complete regular dessins with automorphism group isomorphic to $\mathbf{M}_1(d,e,f)$ is
\[
\nu_1(d,e,f)=
\begin{cases}
1,& \text{if $1\leq d<f=e$},\\
p^{e-f-1}(p-1),&\text{if $1\leq d\leq f<e\leq d+f$},\\
p^{d+e-2f-1}(p-1),&\text{if $1\leq f<d<e\leq d+f$.}
\end{cases}
\]
\end{lemma}

\begin{lemma}\label{ENM2}
Let $d<e$, then for each fixed $f$, $1\leq f<d<e$, up to isomorphism the number $\nu_2(d,e,f)$ of reciprocal
pairs of $(p^d,p^e)$-complete regular dessins with automorphism group isomorphic to $\mathbf{M}_2(d,e,f)$ is
$p^{2d-2f-1}(p-1)$.
\end{lemma}

\begin{lemma}\label{ENM3}
Let $d<e$, then for fixed $h$ and $f$, $h-d\leq f<d<h<e$, up to isomorphism the number $\nu_3(d,e,h,f)$ of
reciprocal pairs of  $(p^d,p^e)$-complete regular dessins  with automorphism group isomorphic to
$\mathbf{M}_3(d,e,h,f)$ is $p^{d+h-2(f+1)}(p-1)^2$.
\end{lemma}

Now we are ready to prove the main result of the paper.

\medskip

\noindent{\bf Proof of Theorem~\ref{MAIN}:} Let $G$ denote the automorphism group of a $(p^d,p^e)$-complete
regular dessin. If $d=0$ then $G$ is a cyclic $p$-group of
order $p^e$, so by Example~\ref{UNIQUE} $\nu(d,e)=1$. If $1\leq d=e$, then by Lemma~\ref{ENM0}, the number
of nonsymmetric dessins is $p^{2(e-1)}-p^{e-1}$, so
\[
\nu(d,e)=p^{e-1}+\frac{1}{2}(p^{2(e-1)}-p^{e-1})=\frac{1}{2}p^{e-1}(1+p^{e-1}).
  \]
In what follows we assume that $1\leq d<e$.
We distinguish four cases:
\begin{case}[1] $1=d<e$. By Theorem~\ref{main} $G\cong\mathbf{M}_1(1,e,f)$ where the numerical
conditions are reduced to $1\leq f\leq e\leq 1+f$. Thus either $f=e-1$ or $f=e$.
By Lemma~\ref{ENM1} $\nu(d,e)=\nu_1(d,e,e-1)+\nu_1(d,e,e)=(p-1)+1=p.$
 \end{case}

 \begin{case}[2] $1<d=e-1$. By Theorem~\ref{main} either $G\cong\mathbf{M}_1(e-1,e,f)$
 where $1\leq f\leq e$, or $G\cong\mathbf{M}_2(e-1,e,f)$ where $1\leq f< e-1$. By Lemma~\ref{ENM1}
 and \ref{ENM2} we get
 \begin{align*}
 \nu(e-1,e)
&=\sum_{1\leq f\leq e}\nu_1(e-1,e,f)+\sum_{1\leq f<e-1}\nu_2(e-1,e,f)\\
&=1+(p-1)+\sum_{1\leq f< e-1}p^{2(e-1)-2f}(p-1)+\sum_{1\leq f<e-1}p^{2(e-1)-2f-1}(p-1)\\
&=p+(p-1)(p^{2e-2}+p^{2e-3})\sum_{1\leq f<e-1}p^{-2f}\\
&=p^{2e-3}.
 \end{align*}
 \end{case}

  \begin{case}[3] $1<d<e-1$ and $e<2d$. Combing the hypothesis with the numerical conditions
   in Theorem~\ref{main} we see that the following subcases may happen:
  (3.1)~$G\cong\mathbf{M}_1(d,e,f)$ where $e-d\leq f\leq e$;
  (3.2)~$G\cong\mathbf{M}_2(d,e,f)$ where $1\leq f<d$;
  (3.3)~$G\cong\mathbf{M}_3(d,e,h,f)$ where $d<h<e$ and $h-d\leq f<d$.
Thus, by Lemma~\ref{ENM1}, ~\ref{ENM2} and ~\ref{ENM3}, we have
\begin{align*}
\nu(d,e)
&=\sum_{e-d\leq f\leq e}\nu_1(d,e,f)+\sum_{1\leq f<d}\nu_2(d,e,f)+\sum_{d<h<e}\sum_{h-d\leq f<d}\nu_3(d,e,h,f)\\
=&1+\sum_{d\leq f<e}p^{e-f-1}(p-1)+\sum_{e-d\leq f<d}p^{d+e-2f-1}(p-1)+\sum_{1\leq f<d}p^{2d-2f-1}(p-1)\\
&+\sum_{d<h<e}\sum_{h-d\leq f<d}p^{d+h-2(f+1)}(p-1)^2\\
=&p^{e-d}+\frac{1}{p+1}(p^{3d-e+1}-p^{e-d+1})+\frac{1}{p+1}(p^{2d-1}-p)+\frac{1}{p+1}
(p^{2d}-p^{3d-e+1}-p^{e-d}+p)\\
=&p^{2d-1}.
\end{align*}
 \end{case}

  \begin{case}[4] $1<d<e-1$ and $e\geq 2d$. By Theorem~\ref{main} the following subcases may happen:
  (4.1)~$G\cong\mathbf{M}_1(d,e,f)$ where $e-d\leq f\leq e$; (4.2)~$G\cong\mathbf{M}_2(d,e,f)$ where $1\leq f<d$;
  (4.3)~$G\cong\mathbf{M}_3(d,e,f)$ where $d<h<2d$ and $h-d\leq f<d$.
Thus, by Lemma~\ref{ENM1}, ~\ref{ENM2} and ~\ref{ENM3}, we have
\begin{align*}
\nu(d,e)
&=\sum_{e-d\leq f\leq e}\nu_1(d,e,f)+\sum_{1\leq f<d}\nu_2(d,e,f)+\sum_{d<h<2d}\sum_{h-d\leq f<d}\nu_3(d,e,f)\\
&=p^d+\frac{1}{p+1}[(p^{2d-1}-p)+(p^{2d}-p^{1+d}-p^d+p)]=p^{2d-1}.
\end{align*}
 \end{case}

Finally,  for each $(p^d,p^e)$-complete regular dessin $\mathcal{D}=(G,\alpha,\beta)$ where
$(\alpha,\beta)$ is an exact $(p^d,p^e)$-bicyclic pair of $G$ given by Lemma~\ref{TRIPLE},
to determine its type $(|\alpha|,|\beta|,|\alpha\beta|)$ it suffices to evaluate $|\alpha\beta|$.
It is easy to check that in all cases $|\alpha\beta|=p^e$. The genus of $\mathcal{D}$
follows from Euler-Poincar\'e formula.\qed

\begin{remark}
Let $(\varphi,\varphi^*)$ be a pair of skew-morphisms $\varphi$ and $\varphi^*$ of the cyclic groups $\mathbb{Z}_n$ and $\mathbb{Z}_m$, and $\pi$ and $\pi^*$ the associated power functions, respectively. The skew-morphism pair $(\varphi,\varphi^*)$ will be called \textit{reciprocal} if they satisfy the following conditions:
\begin{itemize}
\item[\rm(i)] the orders of $\varphi$ and ${\varphi^*}$ divide
    $m$ and $n$, respectively,
\item[\rm(ii)]$\pi(x)=-{\varphi^*}^{-x}(-1)$ and
    ${\pi^*}(y)=-\varphi^{-y}(-1)$ are power functions for
    $\varphi$ and ${\varphi^*}$, respectively.
\end{itemize}
In \cite[Theorem 6]{FHNSW} the authors establish a one-to-one correspondence between the isomorphism
classes of reciprocal pairs of
 $(m,n)$-regular dessins and reciprocal pairs $(\varphi,\varphi^*)$ of skew-morphisms $\varphi$ and $\varphi^*$
of the cyclic groups
$\mathbb{Z}_n$ and $\mathbb{Z}_m$. Therefore, Theorem~\ref{MAIN} also gives rise to
the number of reciprocal pairs of skew-morphisms $\varphi$ and $\varphi^*$ of the cyclic groups
$\mathbb{Z}_{p^d}$ and $\mathbb{Z}_{p^e}$ for odd prime $p$.
\end{remark}

\section*{Acknowledgement}
The first and third author was partially supported by Natural Science Foundation of Zhejiang Province
(No.~LY16A010010, LQ17A010003), and the second author was supported by projects P202/12/G061 of the Czech
Science Foundation and APVV-15-0220 of the Slovak Agency for Research and Development.

 \end{document}